\documentclass[12pt]{amsart}

\usepackage{amsmath}
\usepackage{amsthm}
\usepackage{amssymb}
\usepackage[all]{xy}
\usepackage{mathrsfs}
\usepackage[colorlinks=true]{hyperref}
\usepackage{enumitem}
\usepackage[margin=1in]{geometry}

\newtheorem{theorem}{Theorem}
\newtheorem{lemma}[theorem]{Lemma}
\newtheorem{corollary}[theorem]{Corollary}
\theoremstyle{remark}
\newtheorem*{remark}{Remark}
\theoremstyle{definition}
\newtheorem{definition}[theorem]{Definition}
\numberwithin{theorem}{section}

\def\overnorm#1{\overline{#1}\vphantom{#1}}
\def\undernorm#1{\underline{#1}\vphantom{#1}}

\newcommand{\Hom}{\operatorname{Hom}}
\newcommand{\Spec}{\operatorname{Spec}}
\newcommand{\id}{\mathrm{id}}
\newcommand{\LogSch}{\mathbf{LogSch}}
\newcommand{\Sch}{\mathbf{Sch}}
\newcommand{\et}{\mathrm{\acute{e}t}}

\newcommand{\Mon}{\mathbf{N}}

\makeatletter
\def\getslant#1{\strip@pt\fontdimen1 #1}

\newbox\usefulbox

\def\undernorm#1{\mathchoice
 {{\setbox\usefulbox=\hbox{$\m@th\displaystyle #1$}%
    \dimen@ \getslant\the\textfont\symletters \ht\usefulbox
    \underline{ \box\usefulbox \kern-\dimen@ }\kern\dimen@}}%
 {{\setbox\usefulbox=\hbox{$\m@th\textstyle #1$}%
    \dimen@ \getslant\the\textfont\symletters \ht\usefulbox
    \underline{ \box\usefulbox \kern-\dimen@ }\kern\dimen@}}%
 {{\setbox\usefulbox=\hbox{$\m@th\scriptstyle #1$}%
    \dimen@ \getslant\the\scriptfont\symletters \ht\usefulbox
    \underline{ \box\usefulbox \kern-\dimen@ }\kern\dimen@}}%
 {{\setbox\usefulbox=\hbox{$\m@th\scriptscriptstyle #1$}%
    \dimen@ \getslant\the\scriptscriptfont\symletters \ht\usefulbox
    \underline{ \box\usefulbox \kern-\dimen@ }\kern\dimen@}}%
 {}}
\def\overnorm#1{\mathchoice
 {{\setbox\usefulbox=\hbox{$\m@th\displaystyle #1$}%
    \dimen@ \getslant\the\textfont\symletters \ht\usefulbox
    \kern\dimen@ 
    \overline{\kern-\dimen@ \box\usefulbox } }}
 {{\setbox\usefulbox=\hbox{$\m@th\textstyle #1$}%
    \dimen@ \getslant\the\textfont\symletters \ht\usefulbox
    \kern\dimen@ 
    \overline{\kern-\dimen@ \box\usefulbox } }}
 {{\setbox\usefulbox=\hbox{$\m@th\scriptstyle #1$}%
    \dimen@ \getslant\the\scriptfont\symletters \ht\usefulbox
    \kern\dimen@ 
    \overline{\kern-\dimen@ \box\usefulbox } }}
 {{\setbox\usefulbox=\hbox{$\m@th\scriptscriptstyle #1$}%
    \dimen@ \getslant\the\scriptscriptfont\symletters \ht\usefulbox
    \kern\dimen@ 
    \overline{\kern-\dimen@ \box\usefulbox } }}%
 {}}
\makeatother

\begin{document}

\title{Uniqueness of minimal morphisms of logarithmic schemes}
\author{Jonathan Wise}
\thanks{Supported by an NSA Young Investigator's Grant, Award \#H98230-14-1-0107.}
\date{\today}

\begin{abstract}
We give a sufficient condition under which the moduli space of morphisms between logarithmic schemes is quasifinite under the moduli space of morphisms between the underlying schemes.  This implies that the moduli space of stable maps from logarithmic curves to a target logarithmic scheme is finite over the moduli space of stable maps, and therefore that it has a projective coarse moduli space when the target is projective.
\end{abstract}

\maketitle

\section{Introduction}
\label{sec:intro}

Chen~\cite{Chen}, Abramovich and Chen~\cite{AC}, and Gross and Siebert~\cite{GS} have recently constructed a moduli space of stable maps from logarithmic curves into logarithmic target schemes.  In \cite{minimal}, we extended the existence results of those papers beyond logarithmic curves and eliminated some technical restrictions that applied even to curves.  However, \cite{minimal} only asserts that the moduli space of maps between logarithmic schemes exists as an algebraic space that is locally of finite presentation; Chen showed that when the target has a rank~$1$ logarithmic structure, the moduli space of logarithmic maps is a disjoint union of pieces, each of which is finite over the moduli space of maps between underlying schemes~\cite[Proposition~3.7.5]{Chen}.  This implies in particular that the moduli space of stable maps from logarithmic curves has a projective coarse moduli space when the target variety is projective.  That result was extended to so-called generalized Deligne--Faltings logarithmic structures by Abramovich and Chen~\cite{AC}.

We will explain these results with a general criterion on the domain that applies to arbitrary logarithmic targets:

\begin{theorem} \label{thm:main}
	Let $S$ be a fine logarithmic scheme.  Let $X$ and $Y$ be fine logarithmic $S$-schemes, with $X$ also geometrically reduced, proper, and integral (in the logarithmic sense \cite[Definition~(4.3)]{Kato}) over $S$.  If the relative characteristic monoid of $X/S$ is trivial then the projection from the space of logarithmic maps to the space of maps of underlying schemes factors as an injection followed by an \'etale map
	\begin{equation*}
		\undernorm\Hom_{\LogSch/S}(X,Y) \rightarrow \mathscr T \rightarrow \Hom_{\Sch/\undernorm S}(\undernorm X, \undernorm Y)
	\end{equation*}
	where $\mathscr T$ is the space of \emph{types} (see Definition~\ref{def:type}).
\end{theorem}

The following corollary recovers and generalizes the earlier results of Chen, Abramovich--Chen, and Gross--Siebert, mentioned earlier:

\begin{corollary}
	Let $Y$ be a fine, saturated logarithmic scheme over $S$.  Let $M_u(Y)$ be the moduli space of stable maps from logarithmic curves to $Y$ of type $u$ and let $\undernorm M_u(Y)$ be its underlying algebraic stack and let $M(\undernorm Y)$ be the moduli space of stable maps to $Y$.  Then $\undernorm M_u(Y)$ is finite over $M(\undernorm Y)$.
\end{corollary}
\begin{proof}
We have already seen that the moduli space of stable logarithmic maps is locally of finite type \cite[Theorem~1.1]{minimal}, bounded \cite[Proposition~1.5.6]{logbd}, and satisfies the valuative criterion for properness \cite[Proposition~1.4.3]{logbd} over stable maps, so it remains only to show that the geometric fibers are finite.  Since logarithmic curves have generically trivial relative characteristic monoids, the finiteness is almost immediate from the theorem, which in fact has the stronger conclusion that the map
\begin{equation*}
\undernorm\Hom_{\LogSch/S}(X,Y) \rightarrow \Hom_{\Sch/\undernorm S}(\undernorm X, \undernorm Y)
\end{equation*}
is injective.  However, we are working in the category of fine logarithmic schemes, whereas prevailing convention dictates that the moduli space of stable maps from logarithmic curves to $Y$ be formed in the category of fine, saturated logarithmic schemes.  It is easy to fix this by saturating $\undernorm\Hom_{\LogSch/S}(X,Y)$.  We conclude by remarking that the underlying scheme of the saturation of a fine logarithmic scheme $H$ is always quasifinite over $H$.
\end{proof}

\subsection*{Conventions and notation}

We have retained the notation of \cite{minimal} as much as seemed reasonable.  In particular, $\undernorm X$ is the underlying space or stack of a logarithmic algebraic space or stack $X$.  We have consistently used underlined roman characters to represent schemes, even when no logarithmic scheme is present for the schemes to underlie, but we have not applied the same convention to morphisms of schemes, lest the underlines become overwhelming.

When $A$ and $B$ are objects that vary with objects of some category $\mathscr C$, we write $\Hom_{\mathscr C}(A,B)$ for the functor or fibered category on $\mathscr C$ of morphisms from $A$ to $B$.  This convention conflicts with the more standard convention of using the subscript to indicate in which category the homomorphisms should be taken.  Our perspective is that $\Hom$ should only be applied to pairs of objects of the same type, and that it should be possible to infer the type from the arguments.

\subsection*{Acknowledgements}

This paper was born out of a joint project with D.~Abramovich, Q.~Chen, and S.~Marcus to prove the same result for stable maps of logarithmic curves using a different method.  While my collaborators aver they have not contributed to the present paper and prefer not to be included as coauthors, I could not have written this paper without the intuition gained from the many examples we studied together.  I am very grateful to them for this, and for their generosity in allowing me to take credit for the solution.  I am particularly grateful to Dan Abramovich, whose comments improved the exposition of this paper considerably.

I would also like to thank M.~Chan for a suggestion that was very useful in our original approach to this problem, even though the method presented here circumvents it.

This work was supported by an NSA Young Investigator's grant, award number H98230-14-1-0107.

\section{A preliminary reduction}
\label{sec:reduction}

In order to simplify notation, we make a preliminary reduction.  To prove Theorem~\ref{thm:main}, it is sufficeint to work relative to $\Hom_{\Sch/\undernorm S}(\undernorm X, \undernorm Y)$ and assume that an $\undernorm S$-morphism $f : \undernorm X \rightarrow \undernorm Y$ has already been fixed.  Then we have an equivalence of categories:
\begin{equation} \label{eqn:6}
	\Hom_{\LogSch/S}(X,Y) = \Hom_{\LogSch/S}(f^\ast M_Y, M_X)
\end{equation}
On the left side $\Hom$ should be interpreted as morphisms of logarithmic schemes over $S$; on the right side, $\Hom$ refers to morphisms of logarithmic structures commuting with the structural maps from $\pi^\ast M_S$.

Equation~\eqref{eqn:6} informs us that we can dispense with $Y$ and work entirely on $X$, setting $M = f^\ast M_Y$.  The following assumptions will remain in force for the rest of the paper:
\begin{enumerate}[label=(\roman{*})]
\item $S = (\undernorm S, M_S)$ and $X = (\undernorm X, M_X)$ are fine logarithmic schemes and $\pi : X \rightarrow S$ is a morphism of logarithmic schemes;
\item $M$ is a fine logarithmic structure on $X$, equipped with a homomorphism of logarithmic structures $\pi^\ast M_S \rightarrow M$.
\end{enumerate}

Theorem~\ref{thm:main} breaks up into the following two statements:

\begin{theorem} \label{thm:types}
	Assume that $\undernorm X$ is proper over $\undernorm S$.  Then the space of types (Definition~\ref{def:type}) is \'etale over $\undernorm S$.
\end{theorem}

\begin{theorem} \label{thm:qf}
	Assume that $\undernorm S$ is the spectrum of an algebraically closed field, that $\pi^\ast M_S \rightarrow M_X$ is an integral homomorphism of monoids, that $\undernorm X$ is reduced, and that $\undernorm X$ is proper over $\undernorm S$.  If the sheaf of relative characteristic monoids $M_X / \pi^\ast M_S$ vanishes on a dense open subset of $X$ then for any type $u$ (see Definition~\ref{def:type}), then the underlying algebraic space of $\Hom_{\LogSch/S}(f^\ast M_Y, M)$ has at most one point.
\end{theorem}

The first of these is treated in Section~\ref{sec:types}, and the other takes up the balance of the paper.

\section{Types}
\label{sec:types}

\begin{definition} \label{def:type}
	Let notation be as in Section~\ref{sec:reduction}.  A \emph{type} consists of a homomorphism of sheaves of abelian groups:
	\begin{equation*}
		u : M^{\rm gp} / \pi^\ast M_S^{\rm gp} \rightarrow M_X^{\rm gp} / \pi^\ast M_S^{\rm gp}
	\end{equation*}
\end{definition}

We define:
\begin{align*}
	\mathscr T 
	& = \Hom_{\Sch/\undernorm S}(M^{\rm gp} / \pi^\ast M_S^{\rm gp}, M_X^{\rm gp} / \pi^\ast M_S^{\rm gp})  \\
	& = \Hom_{\Sch/\undernorm S}(\overnorm M^{\rm gp} / \pi^\ast \overnorm M_S^{\rm gp}, \overnorm M_X^{\rm gp} / \pi^\ast \overnorm M_S^{\rm gp})
\end{align*}
More explicitly, for any $\undernorm S$-scheme $\undernorm S'$, write $\undernorm X' = \undernorm X \mathop{\times}_{\undernorm S} \undernorm S'$, and then:
\begin{equation*}
	\mathscr T(S') = \Hom(g^\ast (M^{\rm gp} / \pi^\ast M_S^{\rm gp}), g^\ast (M_X^{\rm gp} / \pi^\ast M_S^{\rm gp})) 
\end{equation*}

\begin{proof}[Proof of Theorem~\ref{thm:types}]
	For the duration of the proof, we abbreviate $F = \overnorm M^{\rm gp} / \pi^\ast \overnorm M_S^{\rm gp}$ and $G = \overnorm M_X^{\rm gp} / \pi^\ast \overnorm M_S^{\rm gp}$.  The object of the proof is to show that $\Hom_{\Sch/\undernorm S}(F,G)$ is representable by an \'etale algebraic space over $\undernorm S$.  Since $F$ and $G$ are constructible, we have
	\begin{equation} \tag{$\ast$} \label{eqn:5}
		\Hom_{\et(\undernorm X')}(g^\ast F, g^\ast G) \simeq g^\ast \Hom_{\et(\undernorm X)}(F, G)
	\end{equation}
	for any morphism $g : \undernorm X' \rightarrow \undernorm X$.  Let us establish~\eqref{eqn:5}:  since these are sheaves it is sufficient to verify it at the stalks, so we can assume that $\undernorm X'$ is the spectrum of an algebraically closed field.  It is sufficient to prove this assertion in an affine neighborhood of each point of $\undernorm X'$, so we can assume $\undernorm X'$ is quasicomapct and quasiseparated; therefore by \cite[Proposition~IX.2.7]{sga4-3}, $F$ has a finite presentation, with $F_0$ and $F_1$ both free:
	\begin{equation*}
		F_1 \rightrightarrows F_0 \rightarrow F
	\end{equation*}
	Then $\Hom_{\et(\undernorm X)}(F,G)$ can be represented as an equalizer:
	\begin{equation*}
		\Hom_{\et(\undernorm X)}(F, G) \rightarrow \Hom_{\et(\undernorm X)}(F_0, G) \rightrightarrows \Hom_{\et(\undernorm X)}(F_1, G) 
	\end{equation*}
	But pullback preserves finite limits, and~\eqref{eqn:5} holds by definition when $F$ is free, so we have what we need:
	\begin{align*}
		g^\ast \Hom_{\et(\undernorm X)}(F,G) 
		& = g^\ast \ker \bigl( \Hom_{\et(\undernorm X)}(F_0, G) \rightrightarrows \Hom_{\et(\undernorm X)}(F_1, G) \bigr) \\
		& = \ker \bigl( \Hom_{\et(\undernorm X)}(g^\ast F_0, G) \rightrightarrows g^\ast \Hom_{\et(\undernorm X)}(g^\ast F_1, G) \bigr) \\
		& = \Hom_{\et(\undernorm X')}(g^\ast F, G)
	\end{align*}
	Thus the espace \'etal\'e of $\Hom_{\et(\undernorm X)}(F,G)$ represents $\Hom_{\Sch/\undernorm X}(F,G)$.

	To complete the proof of the theorem, we now observe that the espace \'etal\'e of $\pi_\ast \Hom_{\et(\undernorm X)}(F,G)$ represents $\Hom_{\Sch/\undernorm S}(F,G)$.  Indeed, for any $g : \undernorm S' \rightarrow \undernorm S$, we have:
	\begin{align*}
		\Hom_{\Sch/\undernorm S}(F,G)(\undernorm S') 
		& = \Hom(g^\ast F, g^\ast G) \\
		& = \Gamma(\undernorm S', \pi_\ast \Hom_{\et(\undernorm X')}(g^\ast F, g^\ast G)) \\
		& = \Gamma(\undernorm S', \pi_\ast g^\ast \Hom_{\et(\undernorm X')}(F, G)) & & \text{by \eqref{eqn:5}} \\
		& = \Gamma(\undernorm S', g^\ast \pi_\ast \Hom_{\et(\undernorm X)}(F,G)) & & \text{by proper base change} \\ & & & \text{\cite[Th\'eor\`eme~5.1~(i)]{sga4-3}} 
	\end{align*}
	This completes the proof.
\end{proof}

\section{The left adjoint to pullback for \'etale sheaves}
\label{sec:adjoint}

Throughout this section, $\pi : \undernorm X \rightarrow \undernorm S$ will be flat and locally of finite presentation, with reduced geometric fibers.  It was shown in \cite[Theorem~4.5]{minimal} that, under these conditions, the pullback functor $\pi^\ast$ for \'etale sheaves has a left adjoint, $\pi_!$.  In this section, we make some further observations about this functor.

We will use the notation $G^{\et}$ for the espace \'etal\'e of an \'etale sheaf~$G$.

\begin{lemma} \label{lem:pi_0}
For any \'etale sheaf $F$ on $\undernorm X$, the fiber of $\pi_! F$ over a geometric point $s$ of $S$ is $\pi_0(X_s)$, where $X_s$ is the fiber of $X$ over $s$.
\end{lemma}
\begin{proof}
Since $\pi_!$ commutes with arbitrary change of base \cite[Corollary~4.5.1]{minimal} we may assume that $\undernorm S$ is the spectrum of an algebraically closed field.  Since $\pi^\ast(G^{\et}) = G^{\et} \mathop\times_{\undernorm S} \undernorm X$ for any \'etale sheaf $G$ on $\undernorm S$ we have:
\begin{equation*}
\Hom(F, \pi^\ast G) = \Hom(F^{\et}, G^{\et} \mathop\times_{\undernorm S} \undernorm X) = \Hom(F^{\et}, G^{\et}) = \Hom(\pi_0(F^{\et}/\undernorm S), G^{\et})
\end{equation*}
Thus $\pi_0(F^{\et}/\undernorm S)$ and $\pi_!(F)^{\et}$ represent the same functor, hence are isomorphic.
\end{proof}

For the next few statements, we will use notation $\pi^{\Mon}_!$ for the left adjoint to $\pi^\ast$ \emph{in the category of integral monoids} \cite[Proposition~4.7]{minimal}, because this functor does not agree with $\pi_!$ upon passage to the underlying sheaf of sets.  In later sections we will only be interested in $\pi_!^{\Mon}$ and not in $\pi_!$, so we will discard the superscript from the former.

The functor $\pi^\ast$ does commute with passage from commutative monoids to their underlying sets, so it follows formally that its left adjoint respects passage from sets to their freely generated monoids:  for any sheaf of sets $F$ on $\undernorm X$, we have
\begin{equation*}
\pi_!^{\Mon}(\mathbf{N} F) = \mathbf{N} \pi_!(F) ,
\end{equation*}
where we have written $\mathbf{N} F$ for the monoid freely generated by $F$.

\begin{lemma} \label{lem:mon-surj}
\begin{enumerate}[label=(\roman{*})]
	\item \label{lem:mon-surj:1} If $F$ is an \'etale sheaf of integral monoids on $\undernorm X$ then $\pi_!^{\Mon} F$ is generated by $\pi_! F$.
\item The functor $\pi_!^{\Mon}$ preserves surjections.
\end{enumerate}
\end{lemma}
\begin{proof}
Let $G \subset \pi^{\Mon}_! F$ be the submonoid generated by $\pi_! F$.  Then $F \rightarrow \pi^\ast \pi_!^{\Mon} F$ factors:
\begin{equation*}
F \rightarrow \pi^\ast \pi_! F \rightarrow \pi^\ast G \subset \pi^\ast \pi_!^{\Mon} F
\end{equation*}
Of course, the first map is just a morphism of sheaves of sets, but the composition is a monoid homomorphism, so upon applying $\pi_!^{\Mon}$ again, we get a commutative diagram:
\begin{equation*} \xymatrix{
\pi_!^{\Mon} F \ar[r] & \pi_!^{\Mon} \pi^\ast G \ar[r] \ar[d] & \pi_!^{\Mon} \pi^\ast \pi_!^{\Mon} F \ar[d] \\
& G \ar[r] & \pi_!^{\Mon} F
} \end{equation*}
Adjunction implies that the composition $\pi_!^{\Mon} F \rightarrow \pi_!^{\Mon} F$ is the identity, from which it follows that $G \rightarrow \pi_!^{\Mon} F$ is surjective.  This proves the first claim.

For the second, consider a surjection $H \rightarrow F$ of sheaves of integral monoids.  Then $\pi_! H \rightarrow \pi_! F$ is surjective, and $\pi_! F$ generates $\pi_!^{\Mon} F$, so the image of $\pi_! H$ in $\pi_!^{\Mon} F$ generates $\pi_!^{\Mon} F$.  Therefore $\pi_!^{\Mon} H \rightarrow \pi_!^{\Mon} F$ is surjective.
\end{proof}

\begin{lemma} \label{lem:gen-gen}
Let $i : \undernorm U \rightarrow \undernorm X$ be the inclusion of an open subset such that $\undernorm U_s \subset \undernorm X_s$ is dense for every geometric point $s$ of $\undernorm S$.  
\begin{enumerate}[label=(\roman{*})]
\item For any \'etale sheaf of sets $F$ on $\undernorm X$, the map $\pi_! i_! i^\ast F \rightarrow \pi_! F$ is surjective.
\item For any \'etale sheaf of integral monoids $F$ on $\undernorm X$, the homomorphism $\pi_!^{\Mon} i_!^{\Mon} i^\ast F \rightarrow \pi_!^{\Mon} F$ is surjective.
\end{enumerate}
\end{lemma}

\begin{proof}
We begin with the statement about sheaves of sets.  Since $\pi_!$ commutes with base change and surjectivity can be checked on the stalks, it is sufficient to assume that $\undernorm S$ is the spectrum of an algebraically closed field.  By Lemma~\ref{lem:pi_0}, we have:
\begin{gather*}
\pi_! F = \pi_0(F^{\et}) \\
\pi_! i_! i^\ast F = \pi_0(i^{-1} F^{\et})
\end{gather*}
But $\undernorm X$ is locally connected (since it is locally of finite type over a field) and $i^{-1} F^{\et} \subset F^{\et}$ is a dense open subset, so $\pi_0(i^{-1} F^{\et})$ surjects onto $\pi_0(F^{\et})$.

Now we prove the statement about monoids.  As before, let $\mathbf N F$ be the sheaf of monoids freely generated by the underlying sheaf of sets of $F$; note that $\mathbf N F \rightarrow F$ is surjective.  Consider the commutative diagram:
\begin{equation*} \vcenter{\xymatrix{
\pi^{\Mon}_! i^{\Mon}_! i^\ast \mathbf N F \ar@{-}[r]^<>(0.5){\sim} \ar[d] & \mathbf N \pi_! i_! i^\ast F \ar[r] & \mathbf N \pi_! F \ar@{-}[r]^<>(0.5){\sim} & \pi^{\Mon}_! \mathbf N F \ar[d] \\
\pi^{\Mon}_! i^{\Mon}_! i^\ast F \ar[rrr] & & & \pi^{\Mon}_! F
}} \end{equation*}
The upper arrow is surjective by the first part of the lemma and Lemma~\ref{lem:mon-surj} implies that the right arrow is surjective, so the bottom arrow is surjective as well.
\end{proof}

\section{Construction of minimal monoids}
\label{sec:summary}

We recall the notation and main construction of \cite{minimal}.  Recall our assumptions from Section~\ref{sec:reduction}, to which we add that $\pi$ is flat with reduced geometric fibers and $\pi^\ast M_S \rightarrow M_X$ is integral.  We also specify a type,~$u$.

We think of fibered categories over $\LogSch$ as fibered categories over $\Sch$ via the projection $\LogSch \rightarrow \Sch$.  Thus, if $H$ is a fibered category and $\undernorm S$ is a scheme, the notation $H(\undernorm S)$ refers to the category of all pairs $(M_S, \xi)$ where $M_S \in \LogSch(\undernorm S)$ is a logarithmic structure on $S$ and $\xi \in H(\undernorm S, M_S)$.

Define $H = \Hom_{\LogSch/S}(M, M_X)$.  This is the fibered category over $\LogSch/S$ whose fiber over $f : (T,M_T)  \rightarrow (S,M_S)$ is the set of morphisms of logarithmic structures
\begin{equation*}
f^\ast M \rightarrow f^\ast M_X
\end{equation*}
on $\undernorm X \mathop{\times}_{\undernorm S} \undernorm T$ that commute with the structural maps from $f^\ast \pi^\ast M_S$ and such that the induced map
\begin{equation*}
f^\ast M^{\rm gp} / f^\ast \pi^\ast M_S^{\rm gp} \rightarrow f^\ast M_X^{\rm gp} / f^\ast \pi^\ast M_S^{\rm gp}
\end{equation*}
coincides with the pullback $f^\ast u$ of the type $u$.  Note that $\undernorm X \mathop\times_{\undernorm S} \undernorm T$ is the underlying scheme of $X \mathop{\times}_S T$ because $X$ is integral over $S$.

We will often wish to refer to objects of the moduli space $H$ with some additional fixed structure.  For example, to fix a logarithmic base scheme $T$, we write $H(T)$; to fix the underlying scheme $\undernorm T$ of $T$ and its characteristic monoid $\overnorm N_T$ (but not the logarithmic structure $N_T$), we write $H(\undernorm T, \overnorm N_T)$.

It was shown in \cite[Theorem~1.1]{minimal} that if $\undernorm X$ is proper over $\undernorm S$ then $H$ is representable by an algebraic space $\undernorm{H}$ with a logarithmic structure $M_H$, that is locally of finite presentation over $\undernorm{S}$.  By a theorem of Gillam, $\undernorm H(\undernorm S) \subset H(\undernorm S)$ may be characterized as the subcategory of \emph{minimal} objects (see \cite{Gillam} or \cite[Appendix~B]{minimal}).  This category is equivalent to a set by \cite[Corollary~5.1.1]{minimal}, which shows minimal objects have no nontrivial automorphisms.

We will be particularly interested in the $\undernorm S$-points of $H$, so we introduce some additional notation for handling them.  For any morphism of logarithmic structures $M_S \rightarrow N_S$, we write $M_X \rightarrow N_X$ for morphism of logarithmic structures on $X$ obtained by pushout along the morphism $\pi^\ast M_S \rightarrow M_X$:  
\begin{equation} \label{eqn:14} \vcenter{\xymatrix{
\pi^\ast M_S \ar[r] \ar[d] & M_X \ar[d] \\
\pi^\ast N_S \ar[r] & N_X
}} \end{equation}
In other words, $(\undernorm X, N_X) = X \mathop{\times}_S (\undernorm S, N_S)$.

Then $H(\undernorm S)$ is the opposite of the category of pairs $(M_S \rightarrow N_S, M \xrightarrow{\varphi} N_X)$ where $N_X$ is as above and and $M \rightarrow N_X$ is a morphism of logarithmic structures commuting with the maps from $\pi^\ast M_S$.  Such an object is generally abbreviated to $(N_S, \varphi)$.  

We now recall the construction of the logarithmic structure of $H$ from \cite{minimal}.  This can be done without explicit reference to the underlying space of $H$:  given an $\undernorm S$-point $(N_S, \varphi)$ of $H$, there is a corresponding map $\alpha : \undernorm S \rightarrow \undernorm H$; there must therefore be a logarithmic structure $\alpha^\ast M_H$ on $\undernorm S$ and a homomorphism of logarithmic structures $\alpha^\ast M_H \rightarrow N_S$; we build the logarithmic structure $\alpha^\ast M_H$.  This is known as the \emph{minimal} (or \emph{basic}) logarithmic structure associated to $(N_S, \varphi)$.

In fact, what we do is find a sheaf of monoids (a \emph{quasilogarithmic structure}, in the language of \cite[Definition~1.2]{minimal}) $\overnorm Q_S$ on $\undernorm S$ such that $\alpha^\ast \overnorm M_H$ is a quotient of $\overnorm Q_S$.  We recall the construction of $\overnorm Q_S$.  

First we form a fiber product:
\begin{equation} \label{eqn:9}
\overnorm R^{\rm gp}_0 = \overnorm M^{\rm gp} \mathop\times_{M_{X/S}^{\rm gp}} \overnorm M_X^{\rm gp}
\end{equation}
The first map in the product is the type $u$ and the second is the tautological projection.  The fiber product comes with two inclusions of $\pi^\ast \overnorm M_S^{\rm gp}$, one from each factor, and we take $\overnorm R^{\rm gp}$ to be their coequalizer.  Then we define the submonoid $\overnorm R \subset \overnorm R^{\rm gp}$ to be the smallest submonoid containing $\pi^\ast \overnorm M_S$ such that the pushout $\overnorm R_X$ in diagram~\eqref{eqn:7}
\begin{equation} \label{eqn:7} \vcenter{ \xymatrix{
\pi^\ast \overnorm M_S \ar[r] \ar[d] & M_X \ar[d] \\
\overnorm R \ar[r] & \overnorm R_X
}} \end{equation}
contains the image of $\overnorm M$ under the tautological map (see \cite[Section~3.1]{minimal} for the construction of this map).  Next, we consider the pushout:
\begin{equation} \label{eqn:8} \vcenter{ \xymatrix{
\pi_! \pi^\ast \overnorm M_S \ar[r] \ar[d] & \pi_! \overnorm R \ar[d] \\
\overnorm M_S \ar[r] & \overnorm Q_S
}} \end{equation}
It was shown in \cite[Section~4.2]{minimal} that if $(N_S, \varphi) \in H(S)$ then there is a unique homomorphism of sheaves of monoids, commuting with maps from $\overnorm M_S$ and inducing $\overnorm\varphi : \overnorm M \rightarrow \overnorm N_X$:
\begin{equation} \label{eqn:12}
\overnorm Q_S \rightarrow \overnorm N_S
\end{equation}
Pulling back the projection $N_S \rightarrow \overnorm N_S$ and the map $N_S \rightarrow \mathcal O_S$ induces a sheaf of monoids $Q_S$ on $\undernorm S$ and a homomorphism $Q_S \rightarrow \mathcal O_S$.  This is not always a logarithmic structure on $S$ (see, for example, the case of logarithmic points in Section~\ref{sec:example}), so we pass to the associated logarithmic structure to get $\alpha^\ast M_H$.

We record several basic consequences of the construction:
\begin{lemma} \label{lem:obs}
\begin{enumerate}[label=(\roman{*})]
\item \label{eqn:2}
$\displaystyle \overnorm Q_S^{\rm gp} / \overnorm M_S^{\rm gp} \simeq \pi_! \overnorm R^{\rm gp} / \pi_! \pi^\ast \overnorm M_S^{\rm gp} \simeq \pi_! (\overnorm R^{\rm gp} / \pi^\ast \overnorm M_S^{\rm gp})$
\item \label{eqn:1}
$\displaystyle \overnorm{R}^{\rm gp} / \pi^\ast \overnorm{M}^{\rm gp}_S \simeq \overnorm{R}_0^{\rm gp} \big/ (\pi^\ast \overnorm M_S^{\rm gp} \times \pi^\ast \overnorm{M}^{\rm gp}_S) \simeq \overnorm{M}^{\rm gp} / \pi^\ast \overnorm{M}_S^{\rm gp}$
\item \label{eqn:10}
The sharpening of $\overnorm Q_S \rightarrow \overnorm N_S$ is $\alpha^\ast \overnorm H_S$.
\item \label{eqn:11}
If $M_{X/S} = 0$ then $\overnorm R = \overnorm M$, canonically.
\end{enumerate}
\end{lemma}
\begin{proof}
Passage to the associated group is a left adjoint and therefore preserves cocartesian diagrams.  Therefore diagram~\eqref{eqn:8} remains cocartesian upon passage to the associated groups.  It is then an easy exercise with universal properties to verify that the quotients along the horizontal direction are isomorphic.  This gives the isomorphism on the left side of~\ref{eqn:2}; for the right side, we observe that $\pi_!$ is a left adjoint, hence respects quotients.

For~\ref{eqn:1}, observe that to go from the middle term to the left one, we quotient both $\overnorm R_0^{\rm gp}$ and $\pi^\ast \overnorm M_S^{\rm gp} \times \pi^\ast \overnorm M_S^{\rm gp}$ by the antidiagonal copy of $\pi^\ast \overnorm M_S^{\rm gp}$.  To go from middle to right, we use the fiber product construction of $\overnorm R_0^{\rm gp}$ in~\eqref{eqn:9} and quotient by the right copy of $\pi^\ast \overnorm M_S^{\rm gp}$.

For~\ref{eqn:10}, first recall that the sharpening of $\overnorm Q_S \rightarrow \overnorm N_S$ is the minimal quotient $\overnorm Q'_S$ of $\overnorm Q_S$ through which the map to $\overnorm N_S$ factors as a sharp homomorphism.  It is constructed by dividing $\overnorm Q_S$ by the set of elements that map to $0$ in $\overnorm N_S$.

Now, $Q_S$ is, by construction, an extension of $\overnorm Q_S$ by $\mathcal O_S^\ast$.  By definition, the associated logarithmic structure $Q_S^a$ of $Q_S$ fits into a cocartesian diagram:
\begin{equation*} \xymatrix{
\exp^{-1} \mathcal O_S^\ast \ar[r] \ar[d] & \mathcal O_S^\ast \ar[d] \\
Q_S \ar[r] & Q_S^a
} \end{equation*}
But note that $\exp^{-1} \mathcal O_S^\ast = \gamma^{-1} N_S^\ast$, where $\gamma : Q_S \rightarrow N_S$ is the tautological map.  Dividing everything by $\mathcal O_S^\ast$ yields another cocartesian diagram:
\begin{equation*} \xymatrix{
\overnorm\gamma^{-1}(0) \ar[r] \ar[d] & 0 \ar[d] \\
\overnorm Q_S \ar[r] & \overnorm Q_S^a
} \end{equation*}
On the other hand, this is precisely the cocartesian diagram used to sharpen $\overnorm Q_S \rightarrow \overnorm N_S$.

For~\ref{eqn:11}, observe that when $M_{X/S} = 0$, the fiber product~\eqref{eqn:9} becomes a product:  $\overnorm R_0^{\rm gp} = \overnorm M^{\rm gp} \times \pi^\ast \overnorm M_S^{\rm gp}$, so equalizing the two copies of $\pi^\ast \overnorm M_S^{\rm gp}$ recovers $\overnorm M^{\rm gp}$.  Chasing the definitions\footnote{The reader who is so inclined may prefer to observe that $\overnorm R$ and $\overnorm M \rightarrow \overnorm R_X$ (where $\overnorm R_X$ is the pushout of $\pi^\ast \overnorm M_S \rightarrow \overnorm M_X$ along $\pi^\ast \overnorm M_S \rightarrow \overnorm R$) satisfy a universal property, and that $\overnorm M$ and $\id : \overnorm M \rightarrow \overnorm M$ visibly satisfy this universal property when $M_{X/S} = 0$.} in \cite[Section~3.1]{minimal}, one discovers that the map $\overnorm M^{\rm gp} \rightarrow \overnorm R_X = \overnorm M^{\rm gp}$ is the identity under this identification, and therefore that $\overnorm R = \overnorm M$.
\end{proof}


\section{A counterexample:  logarithmic points}
\label{sec:example}

It is perhaps easiest to appreciate how the criterion of Theorem~\ref{thm:main} works by studying why quasifiniteness fails in an example where the criterion does not apply.

We consider the moduli space of logarithmic points valued in the standard logarithmic point, considered by Abramovich, Chen, Gillam, and Marcus~\cite{Gillam,ACGM}:  Let $k$ be the spectrum of an algebraically closed field, let $\undernorm S = \Spec k$ and $M_S = 0$, and let $X$ and $Y$ both be the standard logarithmic point over $k$, both regarded as logarithmic schemes over $S$.  The space we are interested in is $\Hom_{\LogSch/S}(X,Y)$.  Since the underlying map of schemes $\undernorm X \rightarrow \undernorm Y$ must be the identity, this may be identified with $\Hom_{\LogSch/S}(M,M_X)$.

We have $M = M_X = \mathbf{N} \times k^\ast$.  An $S$-point of this moduli space is simply a map of logarithmic structures $M \rightarrow M_X$, and our choices are determined by where the generator $(1,1)$ goes.  That gives $\mathbf{N} \times k^\ast$ for the $S$-points.

We could also look at the construction of the minimal monoid associated to one of these maps.  Following the algorithm from Section~\ref{sec:summary}, we should form $\overnorm R^{\rm gp} = \overnorm M^{\rm gp} \mathop{\times}_{M_{X/S}} \overnorm M_{X}^{\rm gp}$.  Since $\overnorm M_S$ is trivial, this is just $\overnorm M^{\rm gp}$.  (There is an additional quotient by $\overnorm M_S^{\rm gp}$ in the algorithm which doesn't change anything since $\overnorm M_S = 0$.)  Then we identify the smallest submonoid $\overnorm R \subset \overnorm R^{\rm gp}$ such that $\overnorm R_X = \overnorm M \times \mathbf{N}$ contains the image of $\overnorm M$ under the tautological map $(\id, u) : \overnorm M \rightarrow \overnorm M^{\rm gp} \times \mathbf{Z}$ (here $u$ is the type, also known as the contact order).  This submonoid is obviously $\overnorm M$ itself, so $\overnorm M$ is the minimal characteristic monoid.

In order to obtain an actual logarithmic structure, we need to assume that $u$ came from an actual logarithmic map over some $(\undernorm S, N_S)$.  This induces a map $\overnorm M \rightarrow \overnorm N_S$ since $\overnorm M$ is minimal on the level of characteristic monoids (see Section~\ref{sec:summary}).  This gives a \emph{quasilogarithmic structure} (see \cite[Definition~1.2]{minimal}) by pulling back $N_S$ to $\overnorm M$, but in order to get a genuine logarithmic structure, we must also sharpen the map $\overnorm M \rightarrow \overnorm N_S$; that is, we must divide $\overnorm M$ by the preimage of $0$.

In our example, we are working over $(S, \mathcal O_S^\ast)$ so that $\overnorm N_S = 0$.  Therefore the sharpening of the minimal quasilogarithmic structure is trivial.

The main observation of this paper was that the sharpening process is responsible for the failure of quasifiniteness.  More precisely, when the sharpening process does not change anything, there is at most one choice (up to unique isomorphism) of a minimal object (at least if the domain is proper).  Indeed, it is a consequence of Corollary~\ref{cor:characteristic} and Lemma~\ref{lem:quotient}, below, that (under an assumption of properness) the automorphism group of the minimal quasilogarithmic structure $Q_S$, as an extension of its characteristic monoid by $\mathcal O_S^\ast$---\emph{but not respecting the map to $\mathcal O_S$}---precisely cancels the choices of maps $M \rightarrow Q_X$.  This automorphism group agrees with the automorphism group of the minimal logarithmic structure exactly when $Q_S$ is the minimal logarithmic structure, and this occurs with the map $\pi_! \overnorm R \rightarrow \overnorm N_S$ is sharp (Lemma~\ref{lem:obs}~\ref{eqn:10}).

Thus the analysis of the fiber of the moduli space of logarithmic maps comes down to the question of whether the minimal quasilogarithmic structure $Q_S$ is already a logarithmic structure.  Lemma~\ref{lem:sharp} shows that a sufficient condition is that the relative characteristic monoid $M_{X/S}$ be generically trivial.

\section{Minimal characteristic monoids}
\label{sec:characteristic}

\begin{lemma} \label{lem:sharp}
Let $\overnorm R$ be constructed as in Section~\ref{sec:summary}.  If the relative characteristic of $X/S$ is generically trivial on every geometric fiber of $X$ over $S$ then for any $(N_S, \varphi) \in H(\undernorm S)$,\footnote{Recall that by our notational conventions, $H(\undernorm S)$ consists of all choices of logarithmic structure $N_S$ on $\undernorm S$ and all $\varphi \in H(\undernorm S, N_S)$.} the canonical map $\pi_! \overnorm R \rightarrow \overnorm N_S$ is sharp.
\end{lemma}
\begin{proof}
Since the construction of $\pi_! \overnorm R$ commutes with change of base, we can reduce to the case where $\overnorm S$ is the spectrum of an algebraically closed field.  By Lemma~\ref{lem:gen-gen}, there is a surjection $\pi_! i_! i^\ast \overnorm R \rightarrow \pi_! \overnorm R$, where $i$ is the inclusion of the dense open subset where $M_{X/S} = 0$.  It is sufficient to show that $\pi_! i_! i^\ast \overnorm R \rightarrow \overnorm N_S$ is sharp.  We can therefore reduce to the case where $M_{X/S} = 0$, globally.

In that case, $\overnorm R = \overnorm M$ (Lemma~\ref{lem:obs}~\ref{eqn:11}) and the map $\pi_! \overnorm R \rightarrow \overnorm N_S$ is induced by adjunction from the map:
\begin{equation*}
\overnorm \varphi : \overnorm M \rightarrow \overnorm N_X = \pi^\ast \overnorm N_S
\end{equation*}
The equality on the right holds because $M_{X/S} = 0$.  But now by Lemma~\ref{lem:pi_0} and Lemma~\ref{lem:mon-surj}~\ref{lem:mon-surj:1}, any $a \in \pi_! \overnorm M$ is the image of a sum of local sections $a_i$ defined over connected $\undernorm U_i$ that are \'etale over $\undernorm X$.  If $a = \sum a_i$ maps to $0$ in $\overnorm N_S$ then all $a_i$ map to $0$ in $\overnorm N_S$ since $\overnorm N_S$ is sharp (as it is the characteristic monoid of a logarithmic structure).  This means $a_i$ maps to $0$ under $\overnorm\varphi : \overnorm M \rightarrow \pi^\ast \overnorm N_S$.  But $\overnorm\varphi$ underlies the morphism of logarithmic structures $\varphi : M \rightarrow \pi^\ast N_S$, hence is sharp.  Therefore all of the $a_i$ must be zero.
\end{proof}

\begin{corollary} \label{cor:characteristic}
Under the hypotheses of the lemma, if $(N_S, \varphi) \in \undernorm H(\undernorm S)$ then $\overnorm N_S = \overnorm Q_S$.
\end{corollary}
\begin{proof}
By Lemma~\ref{lem:obs}~\ref{eqn:10}, $\overnorm N_S$ is the sharpening of $\overnorm Q_S \rightarrow \overnorm N_S$, so the point is to show $\overnorm Q_S \rightarrow \overnorm N_S$ is sharp.  We already know that $\pi_! \overnorm R \rightarrow \overnorm N_S$ is sharp by the lemma.  When $\undernorm X$ has connected geometric fibers over $\undernorm S$, this implies the corollary, since in that case $\pi_! \overnorm R = \overnorm Q_S$.

In general, we can proceed geometric fiber by geometric fiber and assume $\undernorm S$ is the spectrum of an algebraically closed field.  If $\undernorm X = \varnothing$, the conclusion is obvious.  Otherwise, the map $\pi_! \pi^\ast \overnorm M_S \rightarrow \overnorm M_S$ is surjective, which implies that $\pi_! \overnorm R \rightarrow \overnorm Q_S$ is surjective as well.  Therefore the sharpness of $\pi_! \overnorm R \rightarrow \overnorm N_S$ implies the sharpness of $\overnorm Q_S \rightarrow \overnorm N_S$, as required.
\end{proof}

The corollary implies that all objects of $\undernorm H(\undernorm S)$ have the same characteristic monoid.  This greatly simplifies the analysis of $\undernorm H(\undernorm S)$, since the following lemma gives a complete characterization of $H(\undernorm S, \overnorm N_S)$ for any fixed characteristic monoid $\overnorm N_S$:

\begin{lemma} \label{lem:quotient}
Assume $\undernorm S$ is the spectrum of an algebraically closed field.  If $\undernorm H(\undernorm S, \overnorm N_S)$ is nonempty then it is isomorphic to the quotient groupoid\footnote{In fact, the groupoid is a $2$-group and $H(\undernorm S, \overnorm N_S)$ is a pseudotorsor under this $2$-group.}
\begin{equation*}
\bigl[ \Hom(\overnorm{M}^{\rm gp} / \pi^\ast \overnorm{M}^{\rm gp}_S, \mathcal{O}_X^\ast) \big/ \Hom(\overnorm{N}_S^{\rm gp} / \overnorm{M}^{\rm gp}_S, \mathcal{O}_S^\ast) \bigr].
\end{equation*}
\end{lemma}

The map $\Hom(\overnorm N_S / \overnorm M_S, \mathcal O_S^\ast) \rightarrow \Hom(\overnorm M / \pi^\ast \overnorm M_S, \mathcal O_X^\ast)$
used to construct the quotient is obtained from the canonical maps
\begin{gather*}
\pi^\ast \mathcal O_S^\ast \rightarrow \mathcal O_X^\ast \\
\overnorm M^{\rm gp} / \pi^\ast \overnorm M_S^{\rm gp} \rightarrow \overnorm N_X^{\rm gp} / \overnorm M_X^{\rm gp} \simeq \pi^\ast \overnorm N_S^{\rm gp} / \pi^\ast \overnorm M_S^{\rm gp} ,
\end{gather*}
the latter of which is induced from $\overnorm \varphi$ and the cocartesian diagram~\eqref{eqn:14}.

\begin{proof}
We need to see how many ways there are to choose $(N_S, \varphi) \in H_u(\undernorm S)$ with the same fixed characteristic monoid $\overnorm N_S$.  Holding $N_S$ fixed, any two choices of $\varphi$ will differ by a uniquely determined homomorphism $\overnorm{M} \rightarrow \mathcal{O}_X^\ast$ that vanishes on $\pi^\ast \overnorm{M}_S$.  Therefore, for $N_S$ fixed, the choices of $\varphi$ form a torsor under $\Hom(\overnorm{M} / \pi^\ast \overnorm{M}_S, \mathcal{O}_X^\ast)$.

Since $S$ is the spectrum of an algebraically closed field, there is a unique choice of $N_S$ (up to nonunique isomorphism) once $\overnorm{N}_S$ is fixed.  Making such a choice, we can now identify $H_u(\undernorm S, \overnorm N_S)$ with the quotient of $\Hom(\overnorm{M} / \pi^\ast \overnorm{M}_S, \mathcal O_X^\ast)$ by the automorphism group of $N_S$ fixing $\overnorm{N}_S$ and $M_S$.  That automorphism group is precisely $\Hom(\overnorm{N}_S / \overnorm{M}_S, \mathcal O_S^\ast)$, by the same calculation we made above.
\end{proof}

Putting these two lemmas together, we discover first from Lemma~\ref{lem:sharp} that if $H(\undernorm S) \neq \varnothing$ then the characteristic monoid of any object of $\undernorm H(\undernorm S)$ is $\undernorm Q_S$.  Then Lemma~\ref{lem:quotient} implies that $\undernorm H(\undernorm S)$ may be identified with the quotient of $\Hom(\overnorm M^{\rm gp} / \pi^\ast \overnorm M_S^{\rm gp}, \mathcal O_X^\ast)$ by
\begin{align*}
\Hom(\overnorm Q_S^{\rm gp} / \overnorm M_S^{\rm gp}, \mathcal O_S^\ast)  
& = \Hom(\pi_! \overnorm R^{\rm gp} / \pi_! \pi^\ast \overnorm M_S^{\rm gp}, \mathcal O_S^\ast) & & \text{(Lemma~\ref{lem:obs}~\ref{eqn:2})} \\
& = \Hom(\overnorm R^{\rm gp} / \pi^\ast \overnorm M_S^{\rm gp}, \pi^\ast \mathcal O_S^\ast) \\
& = \Hom(\overnorm M^{\rm gp} / \pi^\ast \overnorm M_S^{\rm gp}, \pi^\ast \mathcal O_S^\ast)  & & \text{(Lemma~\ref{lem:obs}~\ref{eqn:1}).}
\end{align*}

We will therefore be able to conclude that the quotient $\undernorm H(\undernorm S)$ is trivial once we prove
\begin{equation*}
\Hom(\overnorm M^{\rm gp} / \pi^\ast \overnorm M_S^{\rm gp}, \mathcal O_X^\ast) = \Hom(\overnorm M^{\rm gp} / \pi^\ast \overnorm M_S^{\rm gp}, \pi^\ast \mathcal O_S^\ast) .
\end{equation*}
That will be done in the next section.

\section{Units}
\label{sec:quasilog}

\begin{theorem} \label{thm:units}
Suppose that $X$ is reduced, that $X$ is proper over $S$, and that $M$ is a coherent logarithmic structure on $X$.  Then any homomorphism $\overnorm M \rightarrow \mathcal O_X^\ast$ factors through $\pi^\ast \mathcal O_S^\ast$.
\end{theorem}

\begin{remark}
	Theorem~\ref{thm:units} is obvious in the special case when $\overnorm M^{\rm gp}$ is generated by global sections, since maps $\overnorm M \rightarrow \mathcal O_X^\ast$ correspond to global sections of $\mathcal O_X^\ast$.
\end{remark}

For any scheme $X$, let $\mathscr U \subset \mathcal O_X^\ast$ be the \'etale subsheaf\footnote{I do not know whether it is necessary to work in the \'etale topology, or if the sheaf is induced from the Zariski topology.} whose sections over an \'etale $V \rightarrow X$ consist of all $f \in \mathcal O_X^\ast(V)$ such that, for every valuation ring $R$ with field of fractions $K$, and every commutative diagram
\begin{equation} \label{eqn:3} \vcenter{ \xymatrix{
			\Spec K \ar[r]^<>(0.5)\varphi \ar[d] & V \ar[d] \\
			\Spec R \ar[r] & X
}} \end{equation}
the restriction $\varphi^\ast f$ of $f$ to $\Spec K$ lies in $R^\ast$.  The idea is that $\mathscr U(V)$ is the sheaf of units in $\mathscr O_X^\ast(V)$ that have no zeroes or poles on the closure of~$V$.

\begin{lemma} \label{lem:U}
Suppose that $M$ is a coherent logarithmic structure on $X$.  Then any homomorphism of sheaves of monoids $\overnorm{M} \rightarrow \mathcal O_X^\ast$ factors through $\mathscr U$.
\end{lemma}
\begin{proof}
	Let $\alpha : \overnorm M \rightarrow \mathcal O_X^\ast$ be a homomorphism.  Consider a section $f \in \overnorm{M}(V)$ for some \'etale $V \rightarrow X$.  We must show that for any commutative diagram~\eqref{eqn:3}, the image of $\varphi^\ast \alpha$ in $K^\ast$ lies in $R^\ast$.  We can therefore replace $X$ with $\Spec R$ and $V$ with $\Spec K$.  Since $K^\ast$ and $R^\ast$ are sheaves in the \'etale topology on $\Spec R$ (namely, $\mathcal O_X^\ast$ and $j_\ast \mathcal O_{\Spec K}^\ast$ where $j$ is the inclusion of the generic point), we can work \'etale-locally in $X$ and assume that $M$ has a global chart.  Then $\alpha$ is the restriction of a global section $\beta$ of $\overnorm M$ and the image of $\beta$ in $\mathcal O_X^\ast$ lies in $\Gamma(X, \mathcal O_X^\ast) = R^\ast$.  Therefore the same holds for $\alpha$.
\end{proof}

\begin{lemma} \label{lem:U-proper}
When $\pi : X \rightarrow S$ is proper and $X$ is reduced, the natural inclusion $\pi^\ast \mathcal O_S^\ast \subset \mathscr U$ is a bijection.
\end{lemma}
\begin{proof}
	Suppose that $f$ is a section of $\mathscr U$ over an \'etale $V \rightarrow X$.  We can assume $V$ is quasicompact.  Let $\overnorm{V}$ be the closure of the graph of $f$ in $X \times \mathbf{P}^1$ (with its reduced structure) and let $\overnorm f : \overnorm V \rightarrow \mathbf P^1$ be the projection.  For any point $q$ of $\overnorm V$, we can choose a valuation of $\mathcal O_{\overnorm V}$ whose center is $q$ and whose generic point lies in $V$; let $R$ be the valuation ring.  Then by definition of $\mathscr U$, the restriction of $\overnorm f$ to $\Spec R$ lies in $R^\ast$.  In particular, $\overnorm f(q) \neq 0, \infty$.  Therefore $\overnorm f$ factors through $\mathbf{G}_m \subset \mathbf{P}^1$.  This holds for any $q \in \overnorm V$, so $\overnorm f \in \Gamma(\overnorm V, \mathcal O_{\overnorm V}^\ast)$.  But $\overnorm V$ is reduced and proper over $S$, so $\Gamma(\overnorm V, \mathcal O_{\overnorm V}^\ast) = \Gamma(\overnorm V, \pi^\ast \mathcal O_S^\ast)$.  But $V$ is reduced (since it is \'etale over $X$ and $X$ is reduced) so $V \subset \overnorm V$ as a scheme and $f$ is the restriction of $\overnorm f$ to $V$.  Thus $f \in \Gamma(V, \pi^\ast \mathcal O_S^\ast)$, as required.
\end{proof}

\begin{proof}[Proof of Theorem~\ref{thm:units}]
	By Lemma~\ref{lem:U}, any homomorphism $\overnorm M \rightarrow \mathcal O_X^\ast$ factors through $\mathscr U$.  But because $X$ is reduced and proper over $S$, Lemma~\ref{lem:U-proper} implies $\mathscr U = \pi^\ast \mathcal O_S^\ast$.
\end{proof}

\begin{proof}[Proof of Theorem~\ref{thm:qf}]
By Corollary~\ref{cor:characteristic}, any $(N_S, \varphi) \in \undernorm H(\undernorm S)$ must have $\undernorm N_S = \undernorm Q_S$.  Then by Lemma~\ref{lem:quotient}, if there are any objects of $H(\undernorm S)$ with characteristic monoid $\undernorm Q_S$ then they are parameterized by the quotient:
\begin{equation} \label{eqn:13}
\bigl[ \Hom(\undernorm M^{\rm gp} / \pi^\ast \overnorm M_S^{\rm gp}, \mathcal O_X^\ast) \big/ \Hom(\undernorm Q_S^{\rm gp} / \overnorm M_S^{\rm gp}, \mathcal O_S^\ast) \bigr]
\end{equation}
The chain of equalities at the end of Section~\ref{sec:characteristic}, along with Theorem~\ref{thm:units}, implies that
\begin{equation*}
\Hom(\overnorm Q_S^{\rm gp} / \overnorm M_S^{\rm gp}, \mathcal O_S^\ast) = \Hom(\overnorm M^{\rm gp} / \pi^\ast \overnorm M_S^{\rm gp}, \mathcal O_X^\ast) 
\end{equation*}
so the quotient~\eqref{eqn:13} is trivial.  
\end{proof}

\bibliographystyle{amsalpha}
\bibliography{minimal}

\end{document}